\documentclass[11pt, reqno]{amsart}
\usepackage{amssymb,mathrsfs,graphicx,extpfeil}
\usepackage{epsfig}
\usepackage{indentfirst, latexsym, amssymb, enumerate,amsmath,graphicx}
\usepackage{float}
\usepackage{colortbl}
\usepackage{epsfig,subfigure}
\usepackage{mathtools}
\usepackage{amsmath}

\title[The inertial Kuramoto model on strongly connected network]{Synchronization of second-order Kuramoto model with frustration on strongly connected digraph}

\author[Zhu]{Tingting Zhu \textsuperscript{\MakeLowercase{a,b}}}

\author[Zhang]{Xiongtao Zhang \textsuperscript{\MakeLowercase{c},*}}

\newtheorem{theorem}{Theorem}[section]
\newtheorem{lemma}{Lemma}[section]

\newtheorem{remark}{Remark}[section]

\newtheorem{definition}{Definition}[section]

\def\charf {\mbox{{\text 1}\kern-.30em {\text l}}}

\begin{document}

\date{\today}


\thanks{\textsuperscript{a} School of Mathematics and Statistics, Hefei University, Hefei, 230601, China (ttzhud201880016@163.com)}
\thanks{\textsuperscript{b} Key Laboratory of Applied Mathematics and Artificial Intelligence Mechanism, Hefei University, Hefei, 230601, China}
\thanks{\textsuperscript{c} School of Mathematics and Statistics, Wuhan University, Wuhan, 430072, China (zhangxt@whu.edu.cn)}
\thanks{\textsuperscript{*} Corresponding author. }

\begin{abstract}
We study the emergent behavior of a second-order Kuramoto-type model with frustration effect on a strongly connected digraph. The main challenge arises from the lack of symmetry in this system, which renders standard approaches for symmetric models, such as the gradient-flow method and classical $\ell^p$ or $\ell^\infty$-type energy estimates, ineffective. To address these difficulties, our primary contribution is the development of time-dependent weighted $\ell^1$-type energy estimates to establish the hypo-coercivity of the frequency diameter. Specifically, we construct novel energy functions incorporating convex combinations of phases, frequencies, accelerations, and jerks, which are shown to be dissipative and capable of bounding both phase and frequency diameters. This framework enables us to demonstrate the emergence of frequency synchronization with an exponential convergence rate.
\end{abstract}
\maketitle \centerline{\date}

\section{Introduction}\label{sec:1}
\vspace{0.5cm}
Synchronization is a widespread phenomenon observed in nature and has been extensively studied across various scientific communities, including engineering, physics, and biology \cite{B-J-Y84, B-B66, E91, E-K91, M-S90}. A variety of mathematical models have been introduced to illustrate this phenomenon, and the rigorous mathematical treatment was pioneered by Winfree \cite{W67} and Kuramoto \cite{K75}. Thereafter, the Kuramoto model has got considerable attention and serve as a prototype for  synchronization. There have been extensive studies on the Kuramoto model from various perspectives \cite{C-H-J-K12, D-H-K20, H-K-L14, H-K-P18, H-K-R16, H-J-K-U20, S-K86, W-S97}. In this paper, we focus on the Kuramoto model with inertia effect and phase shift on a strongly connected network.

To set up the stage, we consider a digraph $G = (V,E)$ composed of a finite vertex set $V= \{1,2,\cdots, N\}$ and an edge set $E \subset V \times V$. For each vertex $i$, we denote by  $\mathcal{N}_i = \{j: (j,i) \in E\}$ the neighbor set of the vertex $i$. The network topology can also be registered by its $(0,1)$-adjacency matrix $(\chi_{ij})$ defined as below:
\begin{equation*}
\chi_{ij} = 
\begin{cases}
\displaystyle 1, \qquad \text{if vertex} \ j \ \text{directly influences vertex} \ i,\\
\displaystyle 0, \qquad \text{otherwise}.
\end{cases}
\end{equation*}
Note that the neighbor set $\mathcal{N}_i$ can also be expressed as $\mathcal{N}_i = \{j : \chi_{ij} > 0\}$. Throughout the paper, we exclude a self loop, i.e., $i \notin \mathcal{N}_i$ for all $1 \le i \le N$.
Kuramoto oscillators are assumed to be located at vertices and interact with each other through the underlying network topology. Let $\theta_i = \theta_i(t)$ and $\Omega_i$ be the phase and natural frequency of the $i$-th Kuramoto oscillators, respectively. Then the dynamics of Kuramoto model with inertia and frustration on a general digraph is given by the following ordinary differential system:
\begin{equation}\label{KMI}
\begin{cases}
\displaystyle m \ddot{\theta}_i(t) + \dot{\theta}_i(t) = \Omega_i  + \kappa \sum_{j \in \mathcal{N}_i} \sin (\theta_j(t) - \theta_i(t) + \alpha), \quad t > 0, \quad i \in V,\\
\displaystyle (\theta_i(0), \omega_i(0)) = (\theta_{i0}, \omega_{i0}),
\end{cases}
\end{equation}
where $\kappa > 0$ represents the coupling strength, and $m > 0, 0 < \alpha < \frac{\pi}{2}$ denotes the inertia and frustration effects, respectively. It is well known that the power network systems are closely related to the second-order Kuramoto-type model, and the transient stability of power grids can be regarded as a synchronization problem for generator rotor angles aiming to restore synchronism after a transient disturbance. Due to the potential application, the rich dynamics of second-order model has been widely investigated in literatures \cite{C-C-C95, C-H-N13, D-B12, D-C-B13, G-W-Y-Z-X-Y20,  L-X-Y14, T-L-O97, Ta-L-O97, W-Q17, W-C21}.

We begin by reviewing recent studies on the second-order Kuramoto model. When the network is symmetric and the frustration is zero, the system is completely symmetric. In this case, the emergence of synchronization can be proven by proper energy methods \cite{C-H-Y11,H-J-K19} or gradient flow techniques \cite{C-L-H-X-Y14, C-L19,L-X-Y14}. However, these methods cannot be directly applied when frustration is introduced, as it breaks the system’s symmetry. In \cite{Ha-K-L14,L-H16}, the authors investigated the frustrated second-order model on a symmetric network and rigorously proved the emergence of the phase-locked states when the frustration is small. Further related works can be found in \cite{D-Z-P-L-J18, F-B-J-K23}. Broadly speaking, these results suggest that a frustrated system on a symmetric network can be viewed as a perturbation of the original symmetric system when the frustration is small. A natural question then arises:
 \vspace{0.2cm}

  Does synchronization occur in \eqref{KMI} when the interaction network is asymmetric?
 \vspace{0.2cm}
 
\noindent Focusing on this question, the authors of the present paper recently investigated the second-order model on an asymmetric network \cite{Z-Z25}. However, the method does not extend to networks with greater depth. In fact, to the best of our knowledge, there is still no rigorous proof establishing synchronization in the second-order Kuramoto model with frustration on a more general network. 
 
The main purpose of the present paper is to demonstrate the emergence of synchronization in the system \eqref{KMI} on a strongly connected network for small inertia and frustration, provided that the initial phase diameter is less than $\pi$. Our motivation is twofold. First, strong connectivity is a typical assumption for the asymmetric networks and serves as a crucial step toward extending our results to more generally connected networks (see Remark  \ref{rem-ge}). Second, it is already clear that synchronization occurs exponentially in the first-order Kuramoto model (whether symmetric or asymmetric) when the oscillators are initially distributed within a half circle \cite{Z-Z23}. Then, it is natural to expect a similar phenomenon in the second-order system \eqref{KMI} for small inertia and frustration.

However, establishing the emergence of synchronization in system \eqref{KMI} presents two main challenges. First, the analysis requires estimates of the evolution of the phase and frequency diameters, which are piecewise smooth but only Lipschitz continuous. As a consequence, their dynamics are not governed by a standard second-order differential equation (or inequality), so that standard analytical approaches are inapplicable. Second, the asymmetry of the interaction network results in non-uniform dissipation of the phase and frequency diameters. For instance, $\dot{D}_\theta(t)$ may be zero at certain times, so that it is difficult to establish a Grownwall type inequality for the phase diameter itself. To address these challenges, we construct energy functions $\mathcal{E}_1$ and $\mathcal{E}_2$ (see \eqref{energy1} and \eqref{energy2}), employing convex combinations of phases, frequencies, accelerations, and jerks (i.e., derivatives of accelerations). These energy functions take the form of weighted-$\ell_1$ norms, represented as $\sum \lambda_i |z_i-z_j|$, where $z_i$ denotes the state variables (including phase, frequency, accelerations, etc.) and the weights $\lambda_i$ depend on the ordering of $z_i$. Through careful estimation, we prove that $\mathcal{E}_2$ governs the frequency diameter and decays exponentially to zero. In other words, by employing appropriately weighted energy estimates, we establish the hypo-coercivity of the frequency diameter and ultimately prove the emergence of synchronization.
 
 \begin{remark}\label{rem-ge}
 We provide some discussions on our approach and result.
\begin{enumerate} 
\item Weighted-$\ell_1$ Energy Method: The weighted-$\ell_1$ energy method employed in this study was initially introduced by the second author and collaborators in \cite{H-L-Z20} to investigate the emergence of flocking in the asymmetric Cucker-Smale model. Subsequently, it was applied to the asymmetric Kuramoto model \cite{Z-Z23}. The present work is the first application of this method to the second-order Kuramoto model, with modifications to the notations in \cite{Z-Z23} for improved readability. We anticipate that it may be applicable to a broader range of models on asymmetric networks. 
 
\item Extension to General Networks: While the current study focuses on the strongly connected case, this serves as a foundational step toward analyzing more general networks. In \cite{H-L-Z20}, the "node-decomposition" method was introduced to decompose a general digraph into multiple strongly connected components. By treating each strongly connected component as a single node, a hierarchical structure among these components emerges. Consequently, once synchronization in \eqref{KMI} is established for strongly connected networks, an inductive approach may facilitate the extension of this result to generally connected networks. This remains an issue for future research. 
 
\end{enumerate}

\end{remark}

The rest of the paper is organized as follows. In Section \ref{sec:2}, we introduce basic notions and the novel construction of four crucial functions as convex combinations, and present our main result. In Section \ref{sec:3}, we present the phase cohesiveness for the inertial Kuramoto model with frustration on a strongly connected digraph. In Section \ref{sec:4}, we give a proof of the main synchronization result. In Section \ref{sec:5}, we provide some numerical simulations. Section \ref{sec:6} concludes with a brief summary of the work.

\section{Preliminaries}\label{sec:2}
\setcounter{equation}{0}
In this section, we first introduce several concepts such as strongly connected digraph and synchronization. Moreover, we construct four important functions in the forum of convex combinations. Our main result on the complete frequency synchronization is also presented.

\subsection{Basic concepts}

In this part, we bring in some concepts on a digraph and synchronization.

\begin{definition}\cite{H-L-Z20} For a digraph $G = (V,E)$ associated to \eqref{KMI}, we have the following definitions.
\begin{enumerate}
\item A path in $G$ from $i_1$ to $i_k$ is a sequence $i_1,i_2,\ldots,i_k$ such that
\begin{equation*}
i_s \in \mathcal{N}_{i_{s+1}} \quad \text{for} \ 1 \le s \le k-1.
\end{equation*}
If there exists a path from $i$ to $j$, then vertex $j$ is said to be reachable from vertex $i$.

\item The digraph $G$ is said to be strongly connected if each vertex can be reachable from any other vertex.
\end{enumerate}
\end{definition}

For system \eqref{KMI}, we denote oscillators' phase, frequency, acceleration and jerk vectors as below
\begin{equation*}
\begin{aligned}
&\theta(t) = (\theta_1(t),\theta_2(t), \ldots, \theta_N(t)), \quad \omega(t) = (\omega_1(t), \omega_2(t), \ldots, \omega_N(t)),\\
&a(t) = (a_1(t), a_2(t), \ldots, a_N(t)), \quad b(t) = (b_1(t), b_2(t),\ldots, b_N(t)),
\end{aligned}
\end{equation*}
where 
\begin{equation}\label{definition_wab}
\omega_i(t) = \dot{\theta}_i(t) , \quad a_i(t) = \dot{\omega}_(t), \quad b_i(t) = \dot{a}_i(t).
\end{equation}
For convenience, we set the corresponding diameters as follows
\begin{equation*}
\begin{aligned}
&D_\theta(t) = \max_{1 \le i \le N} \theta_i(t) - \min_{1 \le i \le N} \theta_i(t), \quad D_\omega(t) = \max_{1 \le i \le N} \omega_i(t) - \min_{1 \le i \le N} \omega_i(t), \\
&D_a(t) = \max_{1 \le i \le N} a_i(t) - \min_{1 \le i \le N} a_i(t), \quad D_b(t) = \max_{1 \le i \le N} b_i(t) - \min_{1 \le i \le N} b_i(t),\\
&\bar{\Omega} = \max_{1 \le i \le N} \Omega_i, \quad \underline{\Omega} = \min_{1 \le i \le N} \Omega_i, \quad D_\Omega = \bar{\Omega} - \underline{\Omega}.
\end{aligned}
\end{equation*}
We next recall the concept of complete frequency synchronization.

\begin{definition}
Let $\theta(t) = (\theta_1(t), \cdots, \theta_N(t))$ be a Kuramoto phase vector associated to system \eqref{KMI}. 

\begin{enumerate}
\item The Kuramoto ensemble asymptotically exhibits complete frequency synchronization if and only if the relative frequency differences tend to zero asymptotically:
\begin{equation*}
\lim_{t \to +\infty} |\omega_i(t) - \omega_j(t)| = 0, \quad \forall \ i \ne j.
\end{equation*}

\item The dynamical state $\theta(t)$ is called an asymptotically phase-locked state if and only if the relative phase differences go to the constant asymptotically:
\begin{equation*}
\lim_{t \to +\infty} |\theta_i(t) - \theta_j(t)| = \theta_{ij}, \quad \forall \ i \ne j.
\end{equation*}

\end{enumerate}

\end{definition}

\subsection{Construction of $Q(t), P(t), A(t)$ and $B(t)$}
To facilitate later analysis, we introduce four functions of convex combination type, which are respectively equivalent to the diameters of phases, frequencies, accelerations, and jerks (the derivatives of accelerations). For this, let $z(t) = \{z_i(t)\}_{i=1}^N$ be the state quantity of an ensemble of $N$ oscillators associated to system \eqref{KMI} and suppose $z_i(t)$ is ordered as follows at time $t$:
\begin{equation*}
z_{s_1}(t) \le z_{s_2}(t) \le \cdots \le z_{s_N}(t),
\end{equation*}
where $s_1s_2\cdots s_N$ is a permutation of $\{1,2,\cdots,N\}$. Then, we follow the idea in  \cite{H-L-Z20} to construct the convex combination of $z_i$. More precisely, for any positive integer $c > 2$, we set 
\begin{equation}\label{coef-convex}
\left\{\begin{aligned}
&M_i=\frac{(c+N-2)!}{(c+N-1-i)!},\quad i=1,2,\cdots, N,\\
&\bar{z}(t)=\frac{\sum\limits_{i=1}^N M_i z_{s_i}}{\sum\limits_{i=1}^N M_i},\\
&\underline{z}(t)=\frac{\sum\limits_{i=1}^N M_i z_{s_{N+1-i}}}{\sum\limits_{i=1}^N M_i},\\
&Z(t)=\bar{z}(t)-\underline{z}(t).
\end{aligned}
\right.
\end{equation}
Note that 
\begin{equation}\label{M_N}
M_1 = 1, \quad M_N = \prod_{i=0}^{N-2} (c+i).
\end{equation}
It is obvious that $\bar{z}$ and $\underline{z}$ are both convex combinations of $z_i(t)$. As the weight coefficient $M_i$ is increasing with respect to $i$, we know that $\bar{z}(t)$ and $\underline{z}(t)$ are approximations to $\max z_i(t)$ and $\min z_i(t)$ respectively. Therefore, $Z(t)=\bar{z}(t)-\underline{z}(t)$ is an approximation to the diameter of $z_i(t)$, which is denoted by $D_z(t)$. In fact, we have the following comparison between $Z(t)$ and $D_z(t)$.

\begin{lemma}\label{Z_Dz_equiv}
Let $z(t) = (z_1(t),\ldots,z_N(t))$ be the state quantity of oscillators associated to system \eqref{KMI} at time $t$. Then, we have
\begin{equation}\label{eta}
\eta D_z(t) \le Z(t) \le D_z(t), \quad \eta = 1 - \frac{4}{c+2},
\end{equation}
where  $D_z(t) = \max\limits_{1 \le i \le N} z_i(t) - \min\limits_{1 \le i \le N} z_i(t)$.
\end{lemma}
\begin{proof}
As $\bar{z}$ and $\underline{z}$ are both convex combinations of $z_i(t)$, it is obvious that $Z(t)\leq D_z(t)$. We now prove the lower bound estimate of $Z(t)$. Without loss of generality, we may assume 
\[z_{1}(t) \le z_{2}(t) \le \cdots \le z_{N}(t).\]
Then, according to the formula of $\bar{z}$, we have 
\begin{align}
\bar{z}-z_1&=\frac{\sum\limits_{i=1}^N M_i z_{i}}{\sum\limits_{i=1}^N M_i}-z_1\geq \frac{M_N (z_{N}-z_1)}{\sum\limits_{i=1}^N M_i}= \frac{ (z_{N}-z_1)}{\sum\limits_{i=1}^{N-2} \frac{M_i}{M_N}+\frac{1}{c}+1}.\label{z-upp}
\end{align}
Then, for $1 \le i\leq N-2$, direct calculation shows that 
\[ \frac{M_i}{M_N}= \frac{(c-1)!}{(c+N-1-i)!} \leq  \frac{1}{(c+N-1-i)(c+N-2-i)}=\frac{1}{c+N-2-i}-\frac{1}{c+N-1-i}.\]
 Substitute above estimates into \eqref{z-upp} to have 
\[\bar{z}-z_1\geq \frac{ (z_{N}-z_1)}{\frac{2}{c}+1}=\frac{c}{c+2}D_z.\]
Similar criteria implies that 
\[z_N-\underline{z}\geq \frac{c}{c+2}D_z.\]
Finally, we combine above estimates together to obtain that 
\[Z=\bar{z}-\underline{z}\geq \frac{2c}{c+2}D_z-D_z=(1-\frac{4}{c+2})D_z.\]
\end{proof}

Subsequently, we apply the construction principle in \eqref{coef-convex} and Lemma \ref{Z_Dz_equiv} to construct four crucial functions $Q(t), P(t), A(t)$ and $B(t)$, which can control the diameters of phases, frequencies, accelerations and jerks, respectively.\\

\noindent $\bullet$ {\bf Construction of $Q(t)$:} For any given time $t \ge 0$, we arrange the oscillators' phases $\{\theta_i(t)\}_{i=1}^N$ from minimum to maximum, i.e.,
\begin{equation*}\label{theta_random_order}
\theta_{s_1}(t) \le \theta_{s_2}(t) \le \cdots \le \theta_{s_N}(t),
\end{equation*}
where $s_1s_2\cdots s_N$ is a permutation of $\{1,2,\cdots,N\}$ depending on the phase values at time $t$. Then we define
\begin{equation}\label{Q_function}
Q(t) = \bar{\theta}(t) - \underline{\theta}(t).
\end{equation}
Moreover, we apply Lemma \ref{Z_Dz_equiv} to have
\begin{equation}\label{Q_Dtheta_equiv}
\eta D_\theta(t) \le Q(t) \le D_\theta(t).
\end{equation}

\noindent $\bullet$ {\bf Construction of $P(t)$:} For any given time $t \ge 0$, the frequencies $\{\omega_i(t)\}_{i=1}^N$ of oscillators are arranged in ascending order, i.e.,
\begin{equation*}\label{w_random_order}
\omega_{l_1} (t) \le \omega_{l_2}(t) \le \cdots\le  \omega_{l_N}(t),
\end{equation*}
where $l_1l_2\cdots l_N$ is a permutation of $\{1,2,\cdots,N\}$ relying on the frequency values at time $t$. Then we define
\begin{equation}\label{P_function}
P(t) = \bar{\omega}(t)- \underline{\omega}(t). 
\end{equation}
Moreover, we have
\begin{equation}\label{P_Dw_equiv}
 \eta D_\omega(t) \le P(t) \le D_\omega(t).
\end{equation}

\noindent $\bullet$ {\bf Construction of $A(t)$:} For any given time $t \ge 0$, the oscillators' accelerations $\{a_i(t)\}_{i=1}^N$ are arranged from minimum to maximum, i.e.,
\begin{equation*}\label{a_random_order}
a_{p_1} (t) \le a_{p_2}(t) \le \cdots\le  a_{p_N}(t),
\end{equation*}
where $p_1p_2\cdots p_N$ is a permutation of $\{1,2,\cdots,N\}$ depending on the acceleration values at time $t$. Then we define
\begin{equation}\label{A_function}
A(t) = \bar{a}(t)- \underline{a}(t).
\end{equation}
Moreover, we have
\begin{equation}\label{A_Da_equiv}
\eta D_a(t) \le A(t) \le D_a(t).
\end{equation}

\noindent $\bullet$ {\bf Construction of $B(t)$:} For any given time $t \ge 0$, we sort the jerks $\{b_i(t)\}_{i=1}^N$ of oscillators in ascending order, i.e.,
\begin{equation*}\label{b_random_order3}
b_{q_1} (t) \le b_{q_2}(t) \le \cdots\le  b_{q_N}(t),
\end{equation*}
where $q_1q_2\cdots q_N$ is a permutation of $\{1,2,\cdots,N\}$ relying on the jerk values at time $t$. Then we define
\begin{equation}\label{tildeB_function}
B(t) = \bar{b}(t)- \underline{b}(t).
\end{equation}
Moreover, we have
\begin{equation}\label{tildeB_Db_equiv}
 \eta D_b(t) \le  B(t) \le D_b(t).
\end{equation}

\subsection{Main result} 
In this part, we present our main result on the emergence of frequency synchronization for the inertial Kuramoto model \eqref{KMI} with frustration on a strongly connected digraph.

\begin{theorem}\label{main}
Let $\theta(t)$ be a solution to system \eqref{KMI} on a strongly connected digraph, and suppose the initial phases satisfy $D_\theta(0) < \pi$. Then for sufficiently large coupling strength $\kappa$ and sufficiently small inertia $m$ and frustration $\alpha$, there exists a finite time $t_* \ge 0$ such that
\begin{equation*}
D_\omega(t) \le Ce^{-\tilde{\Lambda}(t - t_*)}, \quad t \ge t_*,
\end{equation*}
where $C, \tilde{\Lambda}$ are positive constants depending on the system parameters.
\end{theorem}

\begin{remark}
Next, we elaborate on our sufficient framework leading to the complete frequency synchronization. In Theorem \ref{main}, the initial phase diameter is assumed to be less than $\pi$, thus a positive constant $\gamma$ can be found such that
\begin{equation*}\label{initial_con}
D_\theta(0) < \gamma < \pi.
\end{equation*}
The integer $c$ in the convex coefficient setting \eqref{coef-convex} is required to be large such that
\begin{equation}\label{c_set}
c > \max \left\{\frac{N \gamma}{\sin\gamma},\  \frac{N}{\cos(D^\infty+\alpha)} \right\}, \quad D_\theta(0) <  \left(1 - \frac{4}{c+2} \right) (1 -\varepsilon) \gamma,
\end{equation}
where $D^\infty \in \left(0, \frac{\pi}{2} \right)$ is a constant, $D^\infty + \alpha < \frac{\pi}{2}$ is assumed, and $\varepsilon >0$ is a suitable small constant. Recall from \eqref{M_N} and \eqref{eta} that
\begin{equation*}
M_N = \prod_{i=0}^{N-2} (c+i), \quad \eta = 1 - \frac{4}{c+2},
\end{equation*}
which will be determined once $c$ is determined. Then to guarantee the emergence of synchrony, the coupling strength $\kappa$ needs to be enough large, and inertia $m$ and frustration $\alpha$ are required to enough small such that the following constraints are satisfied:
\begin{align}
&m\kappa < \frac{\eta^3}{8N^2 M_N} \min \left\{\frac{\sin \gamma}{ \gamma}, \ \cos (D^\infty + \alpha)\right\}, \label{mk_con1}\\
&D_\theta(0) +  \frac{\eta^2 \sin \gamma}{2N \gamma M_N} mD_\omega(0) +  2m^2 D_a(0) < \eta (1-\varepsilon)\gamma < \pi,\label{mk_con2}\\
&\frac{2(D_\Omega + 2N\kappa \sin \alpha)}{\eta\Lambda} < \min \left\{(1-\varepsilon)\gamma, \ \frac{D^\infty}{2}\ \right\},\label{mk_con3}\\
&4NmM_N\left(D_\omega(0) + D_\Omega\right) + 8 N^2 m\kappa M_N < \eta^2 \cos(D^\infty + \alpha), \label{mk_con4}
\end{align}
where 
\begin{align}
&a_i(0) = \frac{1}{m} \left( - \omega_i(0) + \Omega_i + \kappa \sum_{j \in \mathcal{N}_i} \sin (\theta_j(0) - \theta_i(0) + \alpha ) \right), \ D_a(0) = \max_{1 \le i \le N} a_i(0) - \min_{1 \le i \le N} a_i(0),\label{a}\\
&\Lambda= \min\left(\frac{\eta \kappa \cos \alpha \sin \gamma}{\gamma M_N},\ \frac{1}{m} - \frac{8 N^2 \kappa \gamma M_N}{\eta^3 \sin \gamma},\  \frac{1}{2m}\right). \label{Lambda} 
\end{align}
\end{remark}

\section{Phase cohesinveness}\label{sec:3}
\setcounter{equation}{0}
In this section, we study the phase dynamics of system \eqref{KMI} on a strongly connected digraph. More precisely, we will show that all oscillators will concentrate into an arc confined in a quarter circle. Due to the loss of uniform damping , we intend to construct an energy function which can control the phase diameter, and derive its first-order Gronwall-type inequality, which finally yields the uniformly small boundedness of phase diameter.

\subsection{Energy function and its dynamics}
Based on the analyticity of solution to system \eqref{KMI}, we can divide the time line into a union of countably many intervals:
\begin{equation}\label{divide_time}
[0,+\infty) = \bigcup_{l=1}^{+\infty} I_l, \quad I_l = [t_{l-1},t_l), \quad t_0 = 0,
\end{equation}
such that in each $I_l$, the orders of oscillators' phases, frequencies and accelerations are unchanged. Under the aforementioned settings in Section \ref{sec:2}, we provide a second-order differential inequality of function $Q(t)$ introduced in \eqref{Q_function}. 


\begin{lemma}\label{second_Q_eq}
Let $\theta(t)$ be a solution to system \eqref{KMI} on a strongly connected digraph, and suppose
\begin{equation*}
D_\theta(t) < \gamma < \pi, \quad \text{for} \  t \in I_l,
\end{equation*}
where $I_l$ is some time interval defined in \eqref{divide_time}. Then for enough large $c$ in \eqref{coef-convex}, we have
\begin{equation}\label{second_Q_differential}
m\ddot{Q}(t) + \dot{Q}(t) \le D_\Omega + 2N\kappa \sin \alpha  -   \frac{2\eta \kappa \cos \alpha \sin \gamma}{ \gamma M_N} Q(t), \quad t \in I_l.
\end{equation}
\end{lemma}
\begin{proof}
For $t \in I_l$, we see that there exists some permutation $s_1s_2\cdots s_N$ of $\{1,2,\cdots,N\}$ such that oscillators' phases are sorted in ascending order as below
\begin{equation}\label{theta_random_order2}
\theta_{s_1}(t) \le \theta_{s_2}(t) \le \cdots \le \theta_{s_N}(t).
\end{equation}
For convenience and without loss of generality, we assume $s_i = i$ in \eqref{theta_random_order2}, i.e.,
\begin{equation*}\label{theta_regular_order2}
\theta_1(t) \le \theta_2(t) \le \cdots \le \theta_N(t),
\end{equation*}
if necessary, we still take \eqref{theta_random_order2}. Then, according to the construction of $Q(t)$ in \eqref{Q_function}, we have
\begin{equation}\label{D-1}
m\ddot{Q}(t) + \dot{Q}(t) = m (\ddot{\bar{\theta}}(t) - \ddot{\underline{\theta}}(t)) + (\dot{\bar{\theta}}(t) - \dot{\underline{\theta}}(t)) = m\ddot{\bar{\theta}}(t) + \dot{\bar{\theta}}(t) - (m\ddot{\underline{\theta}}(t) + \dot{\underline{\theta}}(t)).
\end{equation}
Next, we will gain the estimates on the dynamics of $\bar{\theta}$ and $\underline{\theta}$ respectively, and the detailed proof is split into the following two steps.\\

\noindent $\bullet$ {\bf Step 1: (Dynamical estimates of $\bar{\theta}$)} From the definition of $\bar{\theta}$, we have 
\begin{equation}\label{D-1-2}
\bar{\theta}(t)=\frac{\sum\limits_{i=1}^N M_i \theta_{i}}{\sum\limits_{i=1}^N M_i}=\frac{\sum\limits_{i=1}^{N-1} \frac{M_i}{M_N} \theta_{i}+\theta_N}{\sum\limits_{i=1}^N \frac{M_i}{M_N}},\quad \frac{M_i}{M_N}=\frac{1}{\prod\limits_{j=i}^{N-1} (c+N-1-j)},\quad 1\leq i\leq N-1.
\end{equation}
This is a linear combination of $\theta_i$ (a convex combination) and therefore we may obtain the estimate regarding the dynamics of $\bar{\theta}$ by induction. Firstly, we have for $ t \in I_l$
\begin{equation}\label{D-2-1}
\begin{aligned}
m \ddot{\theta}_i(t) + \dot{\theta}_i(t) &= \Omega_i  + \kappa \sum_{j \in \mathcal{N}_i} \sin (\theta_j(t) - \theta_i(t) + \alpha)\\
&\le \bar{\Omega}  + \kappa \sum_{j \in \mathcal{N}_i} [\sin (\theta_j(t) - \theta_i(t)) \cos \alpha + \cos (\theta_j(t) - \theta_i(t)) \sin \alpha]\\
&\le \bar{\Omega} + N\kappa \sin \alpha + \kappa \cos \alpha \min_{\substack{j \in \mathcal{N}_i \\ j\le i}} \sin (\theta_j(t) - \theta_i(t)) + \kappa \sum_{\substack{j \in \mathcal{N}_i\\ j>i}} \sin (\theta_j(t) - \theta_i(t)) \cos \alpha.
\end{aligned}
\end{equation}
Now, we need to control the last positive part in above estimates. We will use method of induction to achieve this. In fact, we apply \eqref{D-2-1} to obtain that 
\begin{equation}\label{D-2}
\begin{aligned}
m \ddot{\theta}_N(t) + \dot{\theta}_N(t) 
&\le \bar{\Omega} + N\kappa \sin \alpha + \kappa \cos \alpha \min_{j \in \mathcal{N}_N} \sin (\theta_j(t) - \theta_N(t)), \quad t \in I_l.
\end{aligned}
\end{equation}
Recall the expression \eqref{D-1-2} and we claim that 
\begin{align}
&\sum_{i=p+1}^{N} \frac{M_i}{M_N} (m\ddot{\theta}_{i}(t)+\dot{\theta}_i(t))\label{D-2-3}\\
&\leq \sum_{i=p+1}^{N} \frac{M_i}{M_N} (\bar{\Omega} + N\kappa \sin \alpha)+  \kappa \cos \alpha \frac{M_{p+1}}{M_N}\sum_{i = p+1}^N  \min_{\substack{k \in \mathcal{N}_i \\ k \le i} } \sin (\theta_k(t) - \theta_i(t)).\notag
\end{align}
According to \eqref{D-2}, the above claim holds for the first step $p=N-1$. Next, we suppose \eqref{D-2-3} holds for $p=j$, and then we verify \eqref{D-2-3} holds for the next step $p= j-1$. For this, we combine \eqref{D-2-1} and \eqref{D-2-3} to have  
\begin{align}
&\sum_{i=j}^{N} \frac{M_i}{M_N} (m\ddot{\theta}_{i}(t)+\dot{\theta}_i(t))\label{D-2-4}\\
&\leq \sum_{i=j}^{N} \frac{M_i}{M_N} (\bar{\Omega} + N\kappa \sin \alpha)+  \kappa \cos \alpha \frac{M_{j}}{M_N}\sum_{i = j}^N  \min_{\substack{k \in \mathcal{N}_i \\ k \le i} } \sin (\theta_k(t) - \theta_i(t))\notag\\
&\quad +\kappa \cos \alpha \frac{M_{j+1}-M_j}{M_N}\sum_{i = j+1}^N  \min_{\substack{k \in \mathcal{N}_i \\ k \le i} } \sin (\theta_k(t) - \theta_i(t)) + \kappa\cos \alpha\frac{M_{j}}{M_N}  \sum_{\substack{k \in \mathcal{N}_j\\ k>j}} \sin (\theta_k(t) - \theta_j(t)).\notag 
\end{align}
We only need to prove the third line of \eqref{D-2-4} is non-positive. It follows fom \eqref{coef-convex} that
\[M_{j+1}=(c+N-1-j)M_j.\]
Then, we recall that $D_\theta(t)<\gamma<\pi$ to have 
\begin{equation}\label{D-2-5}
\begin{aligned}
&(M_{j+1}-M_j)\sum_{i = j+1}^N  \min_{\substack{k \in \mathcal{N}_i \\ k \le i} } \sin (\theta_k(t) - \theta_i(t))\\
&\leq (c-2+N-j)M_j \sum_{i = j+1}^N \min_{\substack{k \in \mathcal{N}_i \\ k \le i} } \frac{\sin \gamma}{\gamma} (\theta_k(t) - \theta_i(t))\\
&\leq (c-2+N-j)M_j \frac{\sin \gamma}{\gamma} (\theta_j-\theta_N),
\end{aligned}
\end{equation}
where the last inequality is due to the strong connectivity of the network and the proof is similar to that in Lemma 4.1 in \cite{Z-Z23}, thus we omit the details. We substitute \eqref{D-2-5} into the third line of \eqref{D-2-4} to have 
\begin{equation}\label{D-2-6}
\begin{aligned}
&\kappa \cos \alpha \frac{M_{j+1}-M_j}{M_N}\sum_{i = j+1}^N  \min_{\substack{k \in \mathcal{N}_i \\ k \le i} } \sin (\theta_k(t) - \theta_i(t)) + \kappa\cos \alpha\frac{M_{j}}{M_N}  \sum_{\substack{k \in \mathcal{N}_j\\ k>j}} \sin (\theta_k(t) - \theta_j(t))\\
&\leq \kappa \cos \alpha \frac{M_j}{M_N}\left((c-2+N-j) \frac{\sin \gamma}{\gamma} (\theta_j(t)-\theta_N(t)) +  \sum_{\substack{k \in \mathcal{N}_j\\ k>j}} \sin (\theta_k(t) - \theta_j(t))\right)\\
&\leq \kappa \cos \alpha \frac{M_j}{M_N}\left((c-2+N-j) \frac{\sin \gamma}{\gamma} (\theta_j(t)-\theta_N(t)) +  (N-j)(\theta_N(t) - \theta_j(t))\right)
\end{aligned}
\end{equation}
As we choose $c>\frac{N\gamma}{\sin \gamma}$ and $1\leq j\leq N-1$, we have
\begin{equation}\label{D-2-7}
\begin{aligned}
&\left((c-2+N-j) \frac{\sin \gamma}{\gamma}  -  (N-j)\right)\\
&= c  \frac{\sin \gamma}{\gamma} -(N-j) -2 \frac{\sin \gamma}{\gamma}+(N-j)\frac{\sin \gamma}{\gamma}\\
&\geq j- \frac{\sin \gamma}{\gamma}\geq 0,
\end{aligned}
\end{equation}
which implies that \eqref{D-2-6} is non-positive. Therefore, we combine \eqref{D-2-4},  \eqref{D-2-6} and \eqref{D-2-7} to complete the verification. More precisely, we conclude \eqref{D-2-3} holds for $j=0,1,\cdots N-1$. Especially when $j=0$, we recall the form of $\bar{\theta}$ in \eqref{D-1-2} and apply the strong connectivity of the network again (similar estimates in \eqref{D-2-5}) to  have
\begin{equation}\label{D-2-8}
\begin{aligned}
m\ddot{\bar{\theta}}(t)+\dot{\bar{\theta}}(t)) &\leq \bar{\Omega} + N\kappa \sin \alpha+  \kappa \cos \alpha \frac{M_{1}}{M_N\sum\limits_{i=1}^N \frac{M_i}{M_N}}\sum_{i = 1}^N  \min_{\substack{k \in \mathcal{N}_i \\ k \le i} } \sin (\theta_k(t) - \theta_i(t))\\
&\leq \bar{\Omega} + N\kappa \sin \alpha +  \kappa \cos \alpha \frac{1}{M_N(1+\frac{2}{c})} \frac{\sin \gamma}{\gamma}(\theta_1(t) - \theta_N(t))\\
&\leq \bar{\Omega} + N\kappa \sin \alpha -  \frac{\eta \kappa \cos \alpha \sin \gamma}{\gamma M_N} D_\theta(t),
\end{aligned}
\end{equation}
where $\eta$ is defined in Lemma \ref{Z_Dz_equiv}.

\noindent $\bullet$ {\bf Step 2: (Dynamical estimates of $\underline{\theta}$ and $Q$)} With similar criteria as in the first step, we can obtain the estimate of dynamics of $\underline{\theta}$ as follows,
\begin{equation}\label{D-2-9}
m\ddot{\underline{\theta}}(t)+\dot{\underline{\theta}}(t)) \geq  \underline{\Omega} - N\kappa \sin \alpha +\frac{\eta \kappa \cos \alpha \sin \gamma}{\gamma M_N} D_\theta(t).
\end{equation}
Now, we recall the definition of $Q$ in \eqref{Q_function} and  combine \eqref{D-2-8} and \eqref{D-2-9} to obtain the following estimate of $Q$
\begin{equation*}
\begin{aligned}
&m\ddot{Q}(t) + \dot{Q}(t) = m\ddot{\bar{\theta}}(t) + \dot{\bar{\theta}}(t) - (m\ddot{\underline{\theta}}(t) + \dot{\underline{\theta}}(t))\\
&\le D_\Omega + 2N\kappa \sin \alpha  -  \frac{2\eta \kappa \cos \alpha \sin \gamma}{\gamma M_N} D_\theta(t)\\
&\le D_\Omega + 2N\kappa \sin \alpha  -  \frac{2\eta \kappa \cos \alpha \sin \gamma}{\gamma M_N} Q(t), \quad t \in I_l.
\end{aligned}
\end{equation*}
\end{proof}

Remark that the above estimate of the dynamics of $Q(t)$ in Lemma \ref{second_Q_eq} is not sufficient to derive the uniform boundedness of phase diameter. The difficulty mainly comes from the second-order derivative of $Q(t)$, which indeed is a convex combination of accelerations $\ddot{\theta}_i$. Hence, we next provide a rough estimate on the dynamics of $A(t)$ defined in \eqref{A_function}, which is equivalent to the diameter of oscillators' accelerations. To this end, we directly differentiate $\eqref{KMI}_1$ with respect to time $t$ to get
\begin{equation}\label{w_KMI}
m \ddot{\omega}_i(t) + \dot{\omega}_i(t) = \kappa \sum_{j \in \mathcal{N}_i} \cos(\theta_j(t) - \theta_i(t) + \alpha)(\omega_j(t) - \omega_i(t)), \quad t > 0, \quad i \in V.
\end{equation}
Recalling $a_i(t) = \dot{\omega}_i(t)$ in \eqref{definition_wab}, one has
\begin{equation}\label{first_a_system}
m \dot{a}_i(t) + a_i(t) = \kappa \sum_{j \in \mathcal{N}_i} \cos(\theta_j(t) - \theta_i(t) + \alpha)(\omega_j(t) - \omega_i(t)).
\end{equation}

\begin{lemma}\label{first_A_eq}
Let $\theta(t)$ be a solution to system \eqref{KMI}. Then, we have
\begin{equation}\label{first_A_differential}
m \dot{A}(t) + A(t) \le \frac{2N\kappa}{\eta} P(t), \quad t \in I_l, \ l=1,2,\cdots,
\end{equation}
where $I_l$ is defined in \eqref{divide_time}.
\end{lemma}
\begin{proof}
We pick out any time interval $I_l$, and for $t \in I_l$,  we can find some permutation $p_1p_2\cdots p_N$ of $\{1,2,\cdots,N\}$ such that oscillators' accelerations are arranged from minimum to maximum as follows
\begin{equation*}\label{a_random_order2}
a_{p_1}(t) \le a_{p_2}(t) \le \cdots \le a_{p_N}(t).
\end{equation*}
According to \eqref{first_a_system}, we have the following simple estimates
\begin{equation}\label{D-3-1}
\begin{aligned}
-N \kappa D_\omega(t)\leq m \dot{a}_i(t) + a_i(t) \leq N \kappa D_\omega(t).
\end{aligned}
\end{equation}
As the estimate \eqref{D-3-1} is independent of $i$, we immediately conclude that the above inequality also holds for  $\bar{a}$ and $\underline{a}$, since they are convex combinations of $a_i$. Then we recall \eqref{A_function} and \eqref{P_Dw_equiv} to have
\begin{equation*}
\begin{aligned}
m \dot{A}(t) + A(t) 
& =  \left(m\dot{\bar{a}}(t) + \bar{a}(t) \right) - \left(m \dot{\underline{a}}(t) + \underline{a}(t) \right)\leq 2N \kappa D_\omega(t) \leq \frac{2N \kappa}{\eta} P(t).
\end{aligned}
\end{equation*}
\end{proof}

Now, we may use the damping term in the differential inequality of $A(t)$ presented in Lemma \ref{first_A_eq} to control the second order derivative of $Q(t)$ in \eqref{second_Q_differential}. However, the dynamics of $Q(t)$ and $A(t)$ are still not dissipative, since there is an additional term $P(t)$ on the right hand side of \eqref{first_A_differential}. This propels us to study the dynamics of $P(t)$ defined in \eqref{P_function}. For this, recall $\omega_i(t) = \dot{\theta}_i(t)$ in \eqref{definition_wab}, and then we see from \eqref{KMI} that
\begin{equation}\label{first_w_system}
m \dot{\omega}_i(t) + \omega_i(t) = \Omega_i + \kappa \sum_{j \in \mathcal{N}_i} \sin (\theta_j(t) - \theta_i(t) + \alpha), \quad t > 0.
\end{equation} 

\begin{lemma}\label{first_P_eq}
Let $\theta(t)$ be a solution to system \eqref{KMI}. Then, we have
\begin{equation}\label{first_P_differential}
m\dot{P}(t) + P(t) \le D_\Omega +2N \kappa \sin \alpha +  \frac{ 2N \kappa \cos \alpha}{\eta} Q(t), \quad t \in I_l, \quad l =1,2,\cdots,
\end{equation}
where $I_l$ is defined in \eqref{divide_time}.
\end{lemma}
\begin{proof}
The proof is similar to Lemma \ref{first_A_eq} and we omit the details.
\end{proof}
In the sequel, we define an energy function
\begin{equation}\label{energy1}
\mathcal{E}_1(t) := Q(t) +  \frac{\eta^2 \sin \gamma }{2N \gamma M_N} mP(t) +  2m^2 A(t), \quad t \ge 0.
\end{equation}
 We will present a Gronwall-type inequality of $\mathcal{E}_1(t)$,  which will govern the time-asymptotic behavior of the function $\mathcal{E}_1(t)$.

\begin{lemma}\label{E1_eq}
Let $\theta(t)$ be a solution to system \eqref{KMI} on a strongly connected digraph. Suppose that the initial phase diameter is less than $\pi$, i.e., there exists a positive constant $\gamma$ such that
\[D_\theta(0) < \gamma <\pi.\]
Then, for sufficiently large $c$ and $\kappa$ and sufficiently small $m$ and $\alpha$ such that 
\[\kappa\gg 1,\quad 0 < m \ll 1,\quad 0<\kappa m\ll1,\quad 0<\alpha\ll 1, \]
 the following differential inequality holds 
\begin{equation}\label{E1_Gronwall}
\frac{d}{dt}\mathcal{E}_1(t) \le  2\left( D_\Omega + 2N\kappa \sin \alpha\right) - \Lambda \mathcal{E}_1(t), \quad \text{a.e.} \ t \ge 0,
\end{equation}
where $\Lambda$ is given in \eqref{Lambda}.
\end{lemma}
\begin{proof}
As $D_\theta(0)<\gamma < \pi$, we apply the continuity of $D_\theta(t)$ to conclude that $D_\theta(t)<\gamma$ in a short time, and thus we can set the following non-empty set  
\begin{equation*}
\mathcal{S} = \{T > 0 \ | \ D_\theta(t) < \gamma < \pi, \ \forall \ 0 \le t < T \},
\end{equation*}
and define $T^* = \sup \mathcal{S}$.  

\noindent $\bullet$ {\bf Step 1:} In the first step, we will prove \eqref{E1_Gronwall} holds before $T^*$. We divide the time interval $[0, T^*)$ into a union of finitely many intervals
\begin{equation*}
[0,T^*) = \bigcup_{l=1}^n J_l, \quad J_l = [\tau_{l-1}, \tau_l), \quad \tau_0 = 0,
\end{equation*}
such that in each time interval $J_l$, the orders of oscillators's phases, frequencies and accelerations are unchanged. Then in any interval $J_l$, we apply Lemma \ref{second_Q_eq}, Lemma \ref{first_A_eq} and Lemma \ref{first_P_eq} to obtain that for $t \in J_l$,
\begin{align}
&m \ddot{Q}(t) + \dot{Q}(t)  \le D_\Omega + 2N\kappa \sin \alpha  -   \frac{2\eta \kappa \cos \alpha \sin \gamma}{\gamma M_N} Q(t), \label{second_Q_diffe2}\\
&m\dot{P}(t) + P(t) \le D_\Omega +2N \kappa \sin \alpha +  \frac{ 2N \kappa \cos \alpha}{\eta} Q(t),\label{first_P_diffe2}\\
&m \dot{A}(t) + A(t) \le \frac{2N\kappa}{\eta} P(t). \label{first_A_diffe2}
\end{align} 
In order to gain the dissipation of the system, we need to combine above inequalities together. We multiply \eqref{first_P_diffe2} by $\frac{\eta^2 \sin \gamma}{2N \gamma M_N}$ to get
\begin{equation}\label{F-1}
\begin{aligned}
\frac{\eta^2 \sin \gamma}{2N \gamma M_N} m\dot{P}(t) + \frac{\eta^2 \sin \gamma}{2N \gamma M_N} P(t) \le \frac{\eta^2 \sin \gamma}{2N \gamma M_N} \left( D_\Omega + 2N\kappa \sin \alpha\right) +  \frac{\eta \kappa \cos \alpha \sin \gamma}{\gamma M_N} Q(t).
\end{aligned}
\end{equation}
Moreover, multiplying \eqref{first_A_diffe2} by $2m$, one has
\begin{equation}\label{F-2}
2m^2 \dot{A}(t) + 2mA(t) \le \frac{4Nm\kappa}{\eta} P(t).
\end{equation}
Then, we add \eqref{second_Q_diffe2}, \eqref{F-1} and \eqref{F-2} together to obtain
\begin{equation*}
\begin{aligned}
& \dot{Q}(t) + \frac{\eta^2 \sin \gamma}{2N \gamma M_N}  m\dot{P}(t)  + 2m^2 \dot{A}(t)  \\
&\le 2(D_\Omega + 2N\kappa \sin \alpha)  - \frac{\eta \kappa \cos \alpha \sin \gamma}{\gamma M_N} Q(t)  -\left( \frac{\eta^2 \sin \gamma}{2N \gamma M_N} - \frac{4Nm\kappa}{\eta} \right)P(t) \\
&-mA(t)-m(A(t) + \ddot{Q}(t)) .
\end{aligned}
\end{equation*} 
Among the convex combination of $a_i(t)$, we know $A(t)$ is the largest one. Therefore, we have $A(t) \geq |\ddot{Q}(t)|$ and above inequality implies 
\begin{equation}\label{F-1-3}
\begin{aligned}
& \dot{Q}(t) + \frac{\eta^2 \sin \gamma}{2N \gamma M_N}  m\dot{P}(t)  + 2m^2 \dot{A}(t)  \\
&\le 2(D_\Omega + 2N\kappa \sin \alpha)  - \frac{\eta \kappa \cos \alpha \sin \gamma}{\gamma M_N} Q(t)  -\left( \frac{\eta^2 \sin \gamma}{2N \gamma M_N} - \frac{4Nm\kappa}{\eta} \right)P(t) -mA(t) .
\end{aligned}
\end{equation} 
Under the assumption of small $m\kappa$,
 above inequality \eqref{F-1-3} provides a dissipation of the energy $\mathcal{E}_1(t)$. More precisely, according to \eqref{mk_con1} and the definition of $\mathcal{E}_1(t)$ in \eqref{energy1},
we apply \eqref{F-1-3} to have 
\begin{equation}\label{F-1-4}
 \frac{d}{dt}\mathcal{E}_1(t) \le 2(D_\Omega + 2N\kappa \sin \alpha)  - \Lambda \mathcal{E}_1(t),\quad t  \in J_l,
\end{equation} 
where $\Lambda$ is defined in \eqref{Lambda}, i.e.,
\begin{equation}\label{F-1-5}
\Lambda= \min\left\{\frac{\eta \kappa \cos \alpha \sin \gamma}{\gamma M_N},\ \frac{1}{m} - \frac{8 N^2 \kappa \gamma M_N}{\eta^3 \sin \gamma},\  \frac{1}{2m}\right\}.
\end{equation}
It is obvious that the energy $\mathcal{E}_1(t)$ is Lipschitz continuous and thus we immediately conclude that \eqref{F-1-4} holds for $[0,T^*) = \bigcup_{l=1}^n J_l$ except for the end point of $J_l$.\newline

\noindent $\bullet$ {\bf Step 2:} In this step, we will use continuity criteria to prove $T^*=+\infty$. We will use proof of contradiction. Suppose  $T^*<+\infty$, then we apply \eqref{F-1-4} to have 
\begin{equation}\label{F-1-6}
\mathcal{E}_1(t)\leq \max\left\{ \mathcal{E}_1(0),\  \frac{2(D_\Omega + 2N\kappa \sin \alpha)}{\Lambda}\right\},\quad t<T^*. 
\end{equation}
On the other hand, we know that
 when $c$ is sufficiently large, the ratio $\eta$ defined in Lemma \ref{Z_Dz_equiv} will be very close to one. 
Therefore, according to $D_\theta(0)<\gamma$ and the formula of $D_a(0)$ in \eqref{a}, we may let $c$ to be sufficiently large, $m\ll1$, $m\kappa \ll 1$ and $\varepsilon\ll 1$ to guarantee that 
\begin{equation}\label{F-1-1}
\mathcal{E}_1(0)\le D_\theta(0) +  \frac{\eta^2 \sin \gamma}{2N \gamma M_N} mD_\omega(0) +  2m^2 D_a(0) < \eta (1-\varepsilon)\gamma < \pi.
\end{equation}
Moreover, according to \eqref{F-1-5}, we know that $\Lambda = \min\{\mathcal{O}(1)\kappa,\ \mathcal{O}(1)\frac{1}{m}\}$, which means 
\begin{equation*}
\frac{2(D_\Omega + 2N\kappa \sin \alpha)}{\Lambda}= \max\left\{\mathcal{O}(1)(\frac{1}{\kappa}+\sin\alpha),\ \mathcal{O}(1)m,\ \mathcal{O}(1)m\kappa \right\}.
\end{equation*}
Therefore, for $\kappa\gg 1$, $\alpha\ll1$, $m\ll1$ and $m\kappa \ll1$, we have 
\begin{equation}\label{F-1-7}
\frac{2(D_\Omega + 2N\kappa \sin \alpha)}{\Lambda} <\eta (1-\varepsilon)\gamma.
\end{equation}
We combine \eqref{F-1-6}, \eqref{F-1-1} and \eqref{F-1-7} to have 
\[\mathcal{E}_1(t)\leq \eta (1-\varepsilon) \gamma, \quad t<T^*.\]
Above estimate immediately implies that $\mathcal{E}_1(T^*) \leq  \eta (1-\varepsilon) \gamma$ due to the continuity of $\mathcal{E}_1(t)$. Then, we recall Lemma \ref{Z_Dz_equiv} and the definition of $\mathcal{E}_1(t)$ to have 
\begin{align*}
D_\theta(T^*)\leq \frac{1}{\eta} Q(T^*) \leq \frac{1}{\eta}\mathcal{E}_1(T^*)\leq (1-\varepsilon) \gamma <\gamma,  \label{F-1-8}
\end{align*}
which is an obvious contradiction to $D_\theta(T^*)=\gamma$ following from the definition of $T^*$.
Therefore, we conclude that $T^*=+\infty$, which finishes the proof.
\end{proof}

\subsection{Entrance to small region} Now, we are ready to show that all Kuramoto oscillators will be trapped into a quarter circle in finite time. 
\begin{lemma}\label{small_region}
Let $\theta(t)$ be a solution to system \eqref{KMI} on a strongly connected digraph. Suppose that the initial phase diameter is less than $\pi$, i.e., there exists a positive constant $\gamma$ such that 
\[D_\theta(0) < \gamma <\pi.\]
Then, for sufficiently large $c$ and $\kappa$ and sufficiently small $m$ and $\alpha$ such that 
\[\kappa\gg 1, \quad 0<\kappa m\ll1,\quad 0<\alpha\ll 1,\]
we can find a finite time $t_*$ such that 
\begin{equation*}
D_\theta(t) < D^\infty < \frac{\pi}{2} , \quad \forall \  t \ge t_*,
\end{equation*}
where $D^\infty < \frac{\pi}{2}$ is a positive constant.
\end{lemma}
\begin{proof}
From \eqref{mk_con3}, \eqref{F-1-5} and the analysis in the previous lemma, we know that 
\begin{equation*}
 \frac{4(D_\Omega + 2N\kappa \sin \alpha)}{\eta\Lambda}= \max\left\{\mathcal{O}(1)(\frac{1}{\kappa}+\sin\alpha),\ \mathcal{O}(1)m,\ \mathcal{O}(1)m\kappa \right\} < D^\infty .
\end{equation*}
which can be acheived when $\kappa$ is large and $\alpha,\ m,\ m\kappa$ are small. 
Moreover, it follows from Lemma \ref{E1_eq} and \eqref{F-1-6} that
\begin{equation}\label{G-1}
\begin{aligned}
\mathcal{E}_1(t)\leq \max\left\{ \mathcal{E}_1(0),\  \frac{2(D_\Omega + 2N\kappa \sin \alpha)}{\Lambda}\right\},\quad t\geq 0.
\end{aligned}
\end{equation}
We then split the proof into the following two cases.

\noindent$\diamond$ {\bf Case 1:} For the case that
\begin{equation*}
\mathcal{E}_1(0) <  \frac{4(D_\Omega + 2N\kappa \sin \alpha)}{\Lambda},
\end{equation*}
we observe from \eqref{G-1} that 
\begin{equation*}
\begin{aligned}
 \mathcal{E}_1(t) \leq \max\left\{ \mathcal{E}_1(0),\  \frac{2(D_\Omega + 2N\kappa \sin \alpha)}{\Lambda}\right\}< \frac{4(D_\Omega + 2N\kappa \sin \alpha)}{\Lambda},\quad t\geq 0.
\end{aligned}
\end{equation*}
Thus, we let $t_*=0$ and obtain that
\begin{equation*}
D_\theta(t)\leq \frac{1}{\eta}Q(t)\leq \frac{1}{\eta}\mathcal{E}_1(t) < \frac{4(D_\Omega + 2N\kappa \sin \alpha)}{\eta\Lambda} <   D^\infty, \quad \text{for} \ t  \ge t_*.
\end{equation*}

\noindent $\diamond$ {\bf Case 2:} For the other case that
\begin{equation*}
\mathcal{E}_1(0) \ge \frac{4(D_\Omega + 2N\kappa \sin \alpha)}{\Lambda},
\end{equation*}
we consider the set 
 \[S=\{t\ |\ \mathcal{E}_1(s)\geq \frac{3(D_\Omega + 2N\kappa \sin \alpha)}{\Lambda},\quad \forall \  0\leq s\leq t\}.\] 
 It is obvious that the set is non-empty since $0\in S$ and $\mathcal{E}_1(t)$ is continuous. Therefore, we may define $t_*=\sup S$. Then, from Lemma \ref{E1_eq}, we have
\begin{equation}\label{G-1-2}
\begin{aligned}
\frac{d}{dt}\mathcal{E}_1(t)& \le  2\left( D_\Omega + 2N\kappa \sin \alpha\right) - \Lambda \mathcal{E}_1(t)\leq -\left( D_\Omega + 2N\kappa \sin \alpha\right),\quad \text{for} \  0 \le t< t_*.
\end{aligned}
\end{equation}
This means $\mathcal{E}_1(t)$ will keep decreasing with a constant rate for $t<t_*$, and from the definition of $t_*$, we have 
\[\mathcal{E}_1(t_*)= \frac{3(D_\Omega + 2N\kappa \sin \alpha)}{\Lambda}.\]
Moreover, similar to \eqref{G-1}, we know that 
\[\mathcal{E}_1(t)\leq  \max\left\{ \mathcal{E}_1(t_*),\  \frac{2(D_\Omega + 2N\kappa \sin \alpha)}{\Lambda}\right\}\le \frac{3(D_\Omega + 2N\kappa \sin \alpha)}{\Lambda},\quad \text{for} \ t\geq t_*.\]
Above analysis means that $\mathcal{E}_1(t)$ will decrease to the interval $[0,\ \frac{3(D_\Omega + 2N\kappa \sin \alpha)}{\Lambda}]$ at $t_*$ and stay in this region for $t\geq t_*$. Now, we integrate \eqref{G-1-2} from $0$ to $t_*$ and obtain that
\[\mathcal{E}_1(t_*)-\mathcal{E}_1(0)\leq -\left( D_\Omega + 2N\kappa \sin \alpha\right)t_*,\]
which implies that 
\[t_*\leq \frac{\mathcal{E}_1(0)-\mathcal{E}_1(t_*)}{\left( D_\Omega + 2N\kappa \sin \alpha\right)}.\]
Therefore, $t_*$ is a finite time and we have for $t\geq t_*$ that
\[D_\theta(t)\leq \frac{1}{\eta}Q(t)\leq \frac{1}{\eta}\mathcal{E}_1(t) \leq  \frac{1}{\eta}\mathcal{E}_1(t_*)=\frac{3(D_\Omega + 2N\kappa \sin \alpha)}{\eta\Lambda}< \frac{4(D_\Omega + 2N\kappa \sin \alpha)}{\eta\Lambda}<   D^\infty.\]
\end{proof}

\section{Frequency synchronization}\label{sec:4}
\setcounter{equation}{0}
In this section, we present that the complete frequency synchronization of system \eqref{KMI} will emerge exponentially fast. Similar to Section \ref{sec:3}, we will introduce an energy function consisting of three functions of convex combination type, which are respectively equivalent to the diameters of frequencies, accelerations and jerks.
Then, we derive one first-order Gronwall-type inequality of this energy function, which ultimately leads to the exponential emergence of frequency synchronization.

To this end, we recall \eqref{first_w_system} and obtain the following rough estimate of $D_\omega(t)$.
\begin{lemma}\label{4-1}
Let $\theta(t)$ be a solution to system \eqref{KMI}. Then, we have 
\[D_\omega(t)\leq D_\omega(0)+D_\Omega+2N\kappa,\quad t\geq 0.\]
\end{lemma}
\begin{proof}
From \eqref{first_w_system}, we immediately obtain the following differential inequality of $D_\omega(t)$
\begin{equation*}
m \dot{D}_\omega(t) + D_\omega(t) \leq  D_\Omega + 2 N \kappa , \quad\text{a.e.}\  t > 0.
\end{equation*} 
It's obvious that $D_\omega(t)$ is Lipschitz continuous and thus we have 
\begin{equation*}
D_\omega(t)  \leq  e^{-\frac{t}{m}}D_\omega(0)+\int_0^t e^{-\frac{t-s}{m}} \frac{(D_\Omega + 2 N \kappa)}{m}ds \leq D_\omega(0)+ D_\Omega + 2 N \kappa.
\end{equation*} 
\end{proof}
Remark that Lemma \ref{4-1} shows a uniform bound of the frequency diameter, but it is insufficient to prove the emergence of synchronization due to the lack of dissipation. In the sequel, we will follow Section \ref{sec:3} to construct a proper energy function to yield the dissipation structure. As we discussed earlier, in Lemma \ref{small_region}, we already proved that $D_\theta(t)<D^\infty<\frac{\pi}{2}$ for $t\geq t_*$. Therefore, in this section, we will study the asymptotic behavior of the solution on $t\in [t_*,+\infty)$. Similarly, due to the analyticity of solution, we split the time interval $[t_*,+\infty)$ into a union of infinitely countable intervals:
\begin{equation}\label{divide_time2}
[t_*, + \infty) = \bigcup_{l=1}^{+\infty} T_l, \quad T_l = [\tilde{t}_{l-1},\tilde{t}_l), \quad \tilde{t}_0 = t_*.
\end{equation}

%


Note that in Lemma \ref{first_P_eq}, we use equation \eqref{first_w_system} to estimate $P(t)$ and obtain differential inequality \eqref{first_P_differential}. However, the dissipation from the damping term in \eqref{first_P_differential}  is not enough to yield the emergence of synchronization. 
Therefore, we intend to consider the following second-order equation of $\omega_i$
\begin{equation}\label{w_KMI2}
m \ddot{\omega}_i(t) + \dot{\omega}_i(t) = \kappa \sum_{j \in \mathcal{N}_i} \cos(\theta_j(t) - \theta_i(t) + \alpha)(\omega_j(t) - \omega_i(t)), \quad t \ge t_*.
\end{equation}
Based on Lemma \ref{small_region} and for small $\alpha$, we see that the right-hand-side of the equation \eqref{w_KMI2} performs like a Cucker-Smale type dissipation,  which implies the following lemma.
\begin{lemma}\label{second_tildeP_eq}
Let $\theta(t)$ be a solution to system \eqref{KMI} on a strongly connected digraph. Moreover, we suppose the conditions in Lemma \ref{small_region} are fulfilled.
Then, 
we have
\begin{equation}\label{second_P_differential}
m \ddot{  P}(t) + \dot{  P}(t)  \le - \frac{2\eta \kappa\cos (D^\infty+\alpha) }{M_N} P(t), \quad t \in T_l, \ l =1,2,\cdots,
\end{equation}
where $T_l$ is defined in \eqref{divide_time2}.
\end{lemma}
\begin{proof}
We pick out any time interval $T_l$, and for $t \in T_l$, we can find some permutation $l_1l_2\cdots l_N$ of $\{1,2,\cdots,N\}$ such that oscillators' frequencies are well-ordered as below
\begin{equation}\label{w_random_order4}
\omega_{l_1}(t) \le \omega_{l_2}(t) \le \cdots \le \omega_{l_N}(t), \quad t \in T_l.
\end{equation}
For convenience, we assume $l_i = i$ in \eqref{w_random_order4} without loss of generality, that is, 
\begin{equation*}\label{w_regular_order4}
\omega_{1}(t) \le \omega_{2}(t) \le \cdots \le \omega_{N}(t), \quad t \in T_l,
\end{equation*}
if necessary, we may still take the form \eqref{w_random_order4}.
Then, from the definition of $  P(t)$ in \eqref{P_function}, we have
\begin{equation}\label{I-1}
m \ddot{  P}(t) + \dot{  P}(t) 
= (m\ddot{\bar{\omega}}(t) +\dot{\bar{\omega}}(t) ) - (m\ddot{\underline{\omega}}(t) + \dot{\underline{\omega}}(t)).
\end{equation}
Next, we split the proof into the following two steps. \\

\noindent $\bullet$ {\bf Step 1: (Dynamical estimates of $\bar{\omega}(t)$)} From the definition of $\bar{\omega}$, we have 
\begin{equation}\label{D-4-2}
\bar{\omega}(t)=\frac{\sum\limits_{i=1}^N M_i \omega_{i}}{\sum\limits_{i=1}^N M_i}=\frac{\sum\limits_{i=1}^{N-1} \frac{M_i}{M_N} \omega_{i}+\omega_N}{\sum\limits_{i=1}^N \frac{M_i}{M_N}},\quad \frac{M_i}{M_N}=\frac{1}{\prod\limits_{j=i}^{N-1} (c+N-1-j)},\quad 1\leq i\leq N-1.
\end{equation}
Based on the result in Lemma \ref{small_region}, in order to estimate the convex combination in \eqref{D-4-2}, we let $\alpha$ be small so that  $D_\theta(t)+\alpha<D^\infty+\alpha<\frac{\pi}{2}$ for $t\geq t_*$. Such $\alpha$ can be found since $D^\infty<\frac{\pi}{2}$. Then, we apply similar calculations as in \eqref{D-2-1} to imply that
\begin{equation}\label{I-2}
\begin{aligned}
m \ddot{\omega}_i(t) + \dot{\omega}_i(t) &= \kappa \sum_{j \in \mathcal{N}_i} \cos (\theta_j(t) - \theta_i(t) + \alpha) (\omega_j(t) - \omega_i(t))\\
&\le \kappa \cos (D^\infty+\alpha) \sum_{\substack{j \in \mathcal{N}_{i} \\ j \le i}} (\omega_j(t) - \omega_i(t)) +\kappa\sum_{\substack{j \in \mathcal{N}_{i} \\ j > i}} (\omega_j(t) - \omega_i(t))\\
&\le \kappa \cos (D^\infty+\alpha) \min_{\substack{j \in \mathcal{N}_{i} \\ j \le i}} (\omega_j(t) - \omega_i(t)) +\kappa\sum_{\substack{j \in \mathcal{N}_{i} \\ j > i}} (\omega_j(t) - \omega_i(t)), \quad t \in T_l.
\end{aligned}
\end{equation}
Similar as in Lemma \ref{second_Q_eq}, we will use method of induction to control the last positive part in above estimates. 
Firstly, when $i=N$, we apply \eqref{I-2} and immediately have that 
\begin{equation}\label{D-4-3}
m \ddot{\omega}_N(t) + \dot{\omega}_N(t) \le \kappa \cos (D^\infty+\alpha) \min_{\substack{j \in \mathcal{N}_{N} \\ j \le N}} (\omega_j(t) - \omega_N(t)) , \quad t \in T_l.
\end{equation}
Now, we claim that
\begin{equation}\label{D-4-4}
\sum_{i=p+1}^{N} \frac{M_i}{M_N} (m\ddot{\omega}_{i}(t)+\dot{\omega}_i(t)) \leq   \kappa  \cos (D^\infty+\alpha) \frac{M_{p+1}}{M_N}\sum_{i = p+1}^N  \min_{\substack{k \in \mathcal{N}_i \\ k \le i} } (\omega_k(t) - \omega_i(t)).
\end{equation}
According to \eqref{D-4-3}, the above claim \eqref{D-4-4} holds for the first step $p=N-1$. Next, we suppose \eqref{D-4-4} holds for $p=j$, and then verify \eqref{D-4-4} holds for the next step $p=j-1$. 
For this, we combine \eqref{I-2} and \eqref{D-4-4} to have  
\begin{align}
&\sum_{i=j}^{N} \frac{M_i}{M_N} (m\ddot{\omega}_{i}(t)+\dot{\omega}_i(t))\label{D-4-5}\\
&\leq  \kappa\cos (D^\infty+\alpha) \frac{M_{j}}{M_N}\sum_{i = j}^N \min_{\substack{k \in \mathcal{N}_i \\ k \le i} } (\omega_k(t) - \omega_i(t))\notag\\
&\quad +\kappa \cos (D^\infty+\alpha) \frac{M_{j+1}-M_j}{M_N}\sum_{i = j+1}^N  \min_{\substack{k \in \mathcal{N}_i \\ k \le i} } (\omega_k(t) - \omega_i(t)) + \kappa \frac{M_{j}}{M_N}  \sum_{\substack{k \in \mathcal{N}_{j} \\ k > j}} (\omega_k(t) - \omega_j(t)).\notag 
\end{align}
Now, as we choose $c >\frac{N}{\cos(D^\infty+\alpha)}$, we can apply similar criterion in \eqref{D-2-5}, \eqref{D-2-6} and \eqref{D-2-7} to prove that the third line of \eqref{D-4-5} is non-positive. Since the proof is almost the same, we omit the details. Therefore, \eqref{D-4-5} implies that 
\[\sum_{i=j}^{N} \frac{M_i}{M_N} (m\ddot{\omega}_{i}(t)+\dot{\omega}_i(t))\leq \kappa\cos (D^\infty+\alpha) \frac{M_{j}}{M_N}\sum_{i = j}^N \min_{\substack{k \in \mathcal{N}_i \\ k \le i} } (\omega_k(t) - \omega_i(t)),\]
which together with the mathematical induction shows that \eqref{D-4-4} holds for $0 \le p \le N-1$, and especially for $p=0$, we follow \eqref{D-4-2} and \eqref{D-2-8} to conclude that 
\begin{equation}\label{D-4-6}
\begin{aligned}
(m\ddot{\bar{\omega}}(t) +\dot{\bar{\omega}}(t) )
&\leq \kappa\cos (D^\infty+\alpha) \frac{M_{1}}{M_N \sum\limits_{i=1}^N \frac{M_i}{M_N}}\sum_{i = 1}^N \min_{\substack{k \in \mathcal{N}_i \\ k \le i} } (\omega_k(t) - \omega_i(t)) \\
&\leq -  \frac{\eta \kappa\cos (D^\infty+\alpha)}{M_N}D_\omega(t).
\end{aligned}
\end{equation}

\noindent $\bullet$ {\bf Step 2: (Dynamical estimates of $\underline{\omega}(t)$ and $P(t)$)} Similar to the arguments in the previous step, we can prove that 
\begin{equation}\label{D-4-7}
\begin{aligned}
(m\ddot{\underline{\omega}}(t) +\dot{\underline{\omega}}(t) ) \geq \frac{\eta \kappa\cos (D^\infty+\alpha) }{M_N} D_\omega(t).
\end{aligned}
\end{equation}
Therefore, we combine \eqref{I-1}, \eqref{D-4-6} and \eqref{D-4-7} to obtain that 
\begin{align*}
m \ddot{  P}(t) + \dot{  P}(t) 
&\leq - \frac{2\eta \kappa\cos (D^\infty+\alpha)}{M_N} D_\omega(t) \leq - \frac{2\eta \kappa\cos (D^\infty+\alpha)}{M_N} P(t),\quad t\in T_l.
\end{align*}
Due to the arbitrary choice of $T_l$, the proof is finished.
\end{proof}

Similarly, the above estimate of the dynamics of $  P(t)$ in Lemma \ref{second_tildeP_eq} does not suffice to conclude the frequency synchronization. The main difficulty lies in the second-order derivative of $  P(t)$ in \eqref{second_P_differential}, which indeed is a convex combination of $\ddot{\omega}_i$. Hence, in what follows, we will provide a rough estimate on the dynamics of $ B(t)$ defined in \eqref{tildeB_function}, which can control the diameter of jerks (the derivatives of accelerations). For this, we observe from \eqref{w_KMI2} that
\begin{equation}\label{first_a_system2}
m\dot{a}_i(t) + a_i(t) = \kappa \sum_{j \in \mathcal{N}_i} \cos (\theta_j(t) - \theta_i(t) + \alpha) (\omega_j(t) - \omega_i(t)), \quad t \ge t_*.
\end{equation}
Directly differentiating on both sides of \eqref{first_a_system2}, we obtain that
\begin{equation*}
\begin{aligned}
m \ddot{a}_i(t) + \dot{a}_i(t) &= -\kappa \sum_{j \in \mathcal{N}_i} \sin (\theta_j(t) - \theta_i(t) + \alpha) (\omega_j(t) - \omega_i(t))^2\\
&\quad+\kappa \sum_{j \in \mathcal{N}_i} \cos (\theta_j(t) - \theta_i(t) + \alpha) (a_j(t) - a_i(t)).
\end{aligned}
\end{equation*}
Moreover, recalling $b_i(t) = \dot{a}_i(t)$, one has
\begin{equation}\label{first_b_system2}
\begin{aligned}
m\dot{b}_i(t) + b_i(t) = -\kappa \sum_{j \in \mathcal{N}_i} \sin (\theta_j - \theta_i + \alpha) (\omega_j - \omega_i)^2+\kappa \sum_{j \in \mathcal{N}_i} \cos (\theta_j - \theta_i + \alpha) (a_j - a_i), \quad t \ge t_*.
\end{aligned}
\end{equation}

\begin{lemma}\label{first_tildeB_eq}
Let $\theta(t)$ be a solution to system \eqref{KMI}. 
Then we have
 \begin{equation*}\label{first_b_differential}
m \dot{ B}(t) +  B(t) \le 2\kappa N (  D_\omega(0)+ D_\Omega + 2 N \kappa) \frac{P(t)}{\eta} + 2\kappa N \frac{A(t)}{\eta}, \quad t \in T_l, \ l=1,2,\cdots,
\end{equation*}
where $T_l$ is defined in \eqref{divide_time2}.
\end{lemma}
\begin{proof}
We pick out any time interval $T_l$, and for $t \in T_l$, there exists some permutation $q_1q_2\cdots q_N$ of $\{1,2,\cdots,N\}$ such that oscillators' jerks are sorted like
\begin{equation*}\label{b_random_order4}
b_{q_1}(t) \le b_{q_2}(t) \le \cdots \le b_{q_N}(t).
\end{equation*}
According to \eqref{tildeB_function}, it is clear to see that
\begin{equation}\label{J-1}
m \dot{ B}(t) +  B(t) 
= \left(m\dot{\bar{b}}(t) + \bar{b}(t) \right) - \left( m\dot{\underline{b}}(t) + \underline{b}(t)\right).
\end{equation}
From \eqref{first_b_system2}, we immediately obtain that 
\[m\dot{\bar{b}}(t) + \bar{b}(t)\leq \kappa N (D_\omega(t))^2 + \kappa N D_a(t),\]
\[m\dot{\underline{b}}(t) + \underline{b}(t)\geq - \kappa N (D_\omega(t))^2 -\kappa N D_a(t).\]
Therefore, we recall Lemma \ref{Z_Dz_equiv} and have 
\begin{align*}
&m \dot{ B}(t) +  B(t) \le 2\kappa N (D_\omega(t))^2 + 2\kappa N D_a(t) \le 2\kappa N D_\omega(t)\frac{P(t)}{\eta} + 2\kappa N \frac{A(t)}{\eta}.
\end{align*}
Finally, we apply the estimate of $D_\omega(t)$ in Lemma \ref{4-1} to obtain that 
  \[    m \dot{ B}(t) +  B(t) \le 2\kappa N (  D_\omega(0)+ D_\Omega + 2 N \kappa) \frac{P(t)}{\eta} + 2\kappa N \frac{A(t)}{\eta}, \quad t \in T_l.\]
 Due to the arbitrary choice of $T_l$, we finishes the proof.
\end{proof}

Now we combine Lemma \ref{second_tildeP_eq}, Lemma \ref{first_tildeB_eq} and Lemma \ref{first_A_eq} to gain the dissipation. More precisely, we introduce the following energy function
\begin{equation}\label{energy2}
\mathcal{E}_2(t) :=   P(t) + \frac{  \eta^2 \cos (D^\infty+\alpha)}{2NM_N} mA(t) + 2m^2 B(t)  , \quad t \ge t_*,
\end{equation}
and will show the exponential decay of $\mathcal{E}_2(t)$. 
\begin{lemma}\label{E2_eq}
Let $\theta(t)$ be a solution to system \eqref{KMI} on a strongly connected digraph. 
Moreover, we suppose the conditions in Lemma \ref{small_region} are fulfilled.
Then, we have
\[D_\omega(t)  \le \frac{\mathcal{E}_2(t_*)}{ \eta}  e^{- \tilde{\Lambda}(t- t_*)}, \quad t \ge t_*,\] 
where $\tilde{\Lambda}$ is a positive constant depending on system parameters.
%
%
%
\end{lemma}
\begin{proof}
From Lemma \ref{second_tildeP_eq}, Lemma \ref{first_tildeB_eq} and Lemma \ref{first_A_eq}, we get that for any time interval $T_l$ and $t \in T_l$ , 
\begin{align}
&m \ddot{  P}(t) + \dot{  P}(t)  \le - \frac{2\eta \kappa\cos (D^\infty+\alpha)}{M_N} P(t), \label{L-1}\\
&m\dot{A}(t) + A(t)   \le \frac{2N\kappa}{ \eta}   P(t),\label{L-2}\\
&m \dot{ B}(t) +  B(t) \le 2\kappa N (  D_\omega(0)+ D_\Omega + 2 N \kappa) \frac{P(t)}{\eta} + 2\kappa N \frac{A(t)}{\eta}. \label{L-3}
\end{align}
We multiply \eqref{L-2} by $\frac{ \eta^2 \cos (D^\infty+\alpha)}{2NM_N}$ to obtain
\begin{equation}\label{L-4}
\frac{ \eta^2 \cos (D^\infty+\alpha) }{2NM_N} m\dot{A}(t) + \frac{ \eta^2 \cos (D^\infty+\alpha) }{2NM_N} A(t) \le  \frac{\eta \kappa\cos (D^\infty+\alpha)}{M_N}   P(t).
\end{equation}
Moreover, multiplying \eqref{L-3} by $2m$, one has
\begin{equation}\label{L-5}
2m^2\dot{ B}(t) + 2m  B(t) \le 4m\kappa N (  D_\omega(0)+ D_\Omega + 2 N \kappa) \frac{P(t)}{\eta} + 4m\kappa N \frac{A(t)}{\eta}.
\end{equation}
Then, we add \eqref{L-1}, \eqref{L-4} and \eqref{L-5} together to get
\begin{equation*}
\begin{aligned}
&m \ddot{  P}(t) + \dot{  P}(t) +\frac{ m \eta^2 \cos (D^\infty+\alpha) }{2NM_N} \dot{A}(t) +\frac{ \eta^2 \cos (D^\infty+\alpha) }{2NM_N}  A(t) + 2m^2\dot{ B}(t) + 2m  B(t)\\
&\le - \frac{\eta\kappa\cos (D^\infty+\alpha)}{M_N}  P(t) +  \frac{4m\kappa N (  D_\omega(0)+ D_\Omega + 2 N \kappa)}{\eta}   P(t) + \frac{4Nm\kappa}{ \eta} A(t).
\end{aligned}
\end{equation*}
This further implies that
\begin{equation}\label{D-4-8}
\begin{aligned}
&\frac{d}{dt} \left(   P(t) + \frac{ m \eta^2 \cos (D^\infty+\alpha)}{2NM_N} A(t) + 2m^2 B(t)\right)\\
&\le - m(B(t) + \ddot{  P}(t)) - \left(\frac{  \eta^2 \cos (D^\infty+\alpha) }{2NM_N}- \frac{4Nm\kappa}{ \eta} \right)  A(t) - m  B(t)\\
&\quad - \left( \frac{\eta\kappa\cos (D^\infty+\alpha)}{M_N} - \frac{4m\kappa N (  D_\omega(0)+ D_\Omega + 2 N \kappa) }{\eta} \right)  P(t) .
\end{aligned}
\end{equation}
Due to the definition of $B(t)$ and $P(t)$, it is obvious that $B(t)\geq |\ddot{P}(t)|$. Then, we combine \eqref{D-4-8} and the similar analysis as in \eqref{F-1-3} and \eqref{F-1-5} to obtain that 
\begin{equation}\label{D-4-9}
\begin{aligned}
\frac{d}{dt}\mathcal{E}_2(t) &= \frac{d}{dt} \left(   P(t) + \frac{\eta^2 \cos (D^\infty+\alpha)}{2NM_N} mA(t) + 2m^2 B(t)\right) \\
&\leq - \tilde{\Lambda} \left(   P(t) + \frac{\eta^2 \cos (D^\infty+\alpha) }{2NM_N} mA(t) + 2m^2 B(t)\right) = - \tilde{\Lambda}  \mathcal{E}_2(t),
 \end{aligned}
\end{equation} 
where $\tilde{\Lambda}$ is given by
\begin{equation}\label{D-4-10}
\tilde{\Lambda}=\min\left\{\left(\frac{\eta\kappa\cos (D^\infty+\alpha) }{M_N} - \frac{4m\kappa N (  D_\omega(0)+ D_\Omega + 2 N \kappa) }{\eta} \right), \ \frac{\left(\frac{  \eta^2 \cos (D^\infty+\alpha)}{2NM_N}- \frac{4Nm\kappa}{ \eta} \right)}{\frac{ m \eta^2 \cos (D^\infty+\alpha) }{2NM_N} },\ \frac{1}{2m}\right\}.
\end{equation} 
As we require $m\kappa \ll 1$, \eqref{D-4-10} is non-negative, 
and thus \eqref{D-4-9} is dissipative. Then we recall Lemma \ref{Z_Dz_equiv} to obtain that 
\begin{equation*}\label{L-7}
D_\omega(t) \le \frac{1}{ \eta}   P(t) \le \frac{1}{ \eta}  \mathcal{E}_2(t) \le \frac{1}{ \eta} \mathcal{E}_2(t_*) e^{- \tilde{\Lambda} (t- t_*)}, \quad t \ge t_*.
\end{equation*}
\end{proof}

Finally, we give the proof of Theorem \ref{main}.\\

\noindent {\bf Proof of Theorem \ref{main}:} According to Lemma \ref{small_region}, under the assumptions of sufficiently large $c$ and $\kappa$ and sufficiently small $m$ and $\alpha$, we derive that there exists a finite time $t_* \ge 0$ such that
\begin{equation*}
D_\theta(t) < D^\infty < \frac{\pi}{2} , \quad \forall \  t \ge t_*.
\end{equation*}
Based on this fine estimate, the proof for the frequency synchronization result in Theorem \ref{main} is completed following from Lemma \ref{E2_eq}.

\section{Simulations}\label{sec:5}
\setcounter{equation}{0}
In this section, we present a simple example with three oscillators to perform som simulations for the synchronization result in Theorem \ref{main}. We employ the classical numerical method of fourth order Runge-Kutta. The interaction network $(\chi_{ij})$ is given as below
\begin{equation}\label{P-1}
(\chi_{ij}) = 
\begin{pmatrix}
0 & 0 & 1\\
1 & 0 & 0\\
0 & 1 & 0
\end{pmatrix}
,
\end{equation}
which is strongly connected (that is, $1 \to 2 \to 3 \to 1$).

The initial phases are randomly chosen with $D_\theta(0) = 1.0330$, and we can set $\gamma = 1.8955$ such that $D_\theta(0) < \gamma < \pi$. The initial frequencies are randomly chosen in $(-0.5, 0.5)$ with $D_\omega(0) = 0.6080$, the natural frequencies in $(-10^{-4}, 10^{-4})$ with $D_\Omega = 1.5\times 10^{-4}$. 
Moreover, we choose $D^\infty = 0.4, \ \alpha = 10^{-5} $ and $\varepsilon = 10^{-3}$ which satisfy $D^\infty + \alpha < \frac{\pi}{2}$. Then, the integer parameter $c$ in the setting \eqref{coef-convex} of convex coefficients can be determined by the relation \eqref{c_set}, and we set $c = 7$. 

Based on above choices and estimates, we set the coupling strength $K =1$ and the inertia $m = 10^{-5}$ so that the conditions \eqref{mk_con1}-\eqref{mk_con4} are satisfied 
Then, the convergence to the synchronized behavior is shown in Figure \ref{Fig1}. For ease of observation, we plot the formation in the local time interval $[12,15]$ in Figure \ref{Fig1:local_phs1} and \ref{Fig1:local_fre1} for Figure \ref{Fig1:phs1} and \ref{Fig1:fre1}, respectively.

\begin{figure}[h]
\centering
\mbox{
\subfigure[Phase-locked state]{\includegraphics[width=0.48\textwidth]{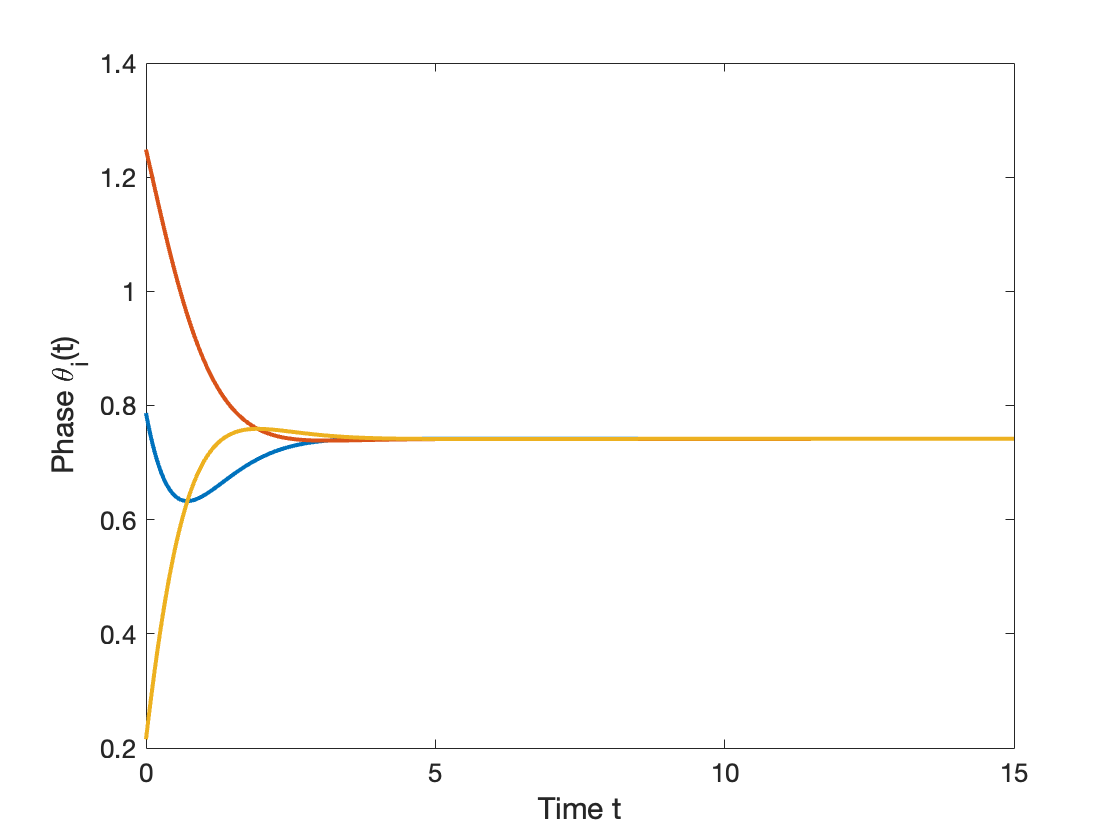}\label{Fig1:phs1}}

\subfigure[Phase-locked state on local time interval]{\includegraphics[width=0.48\textwidth]{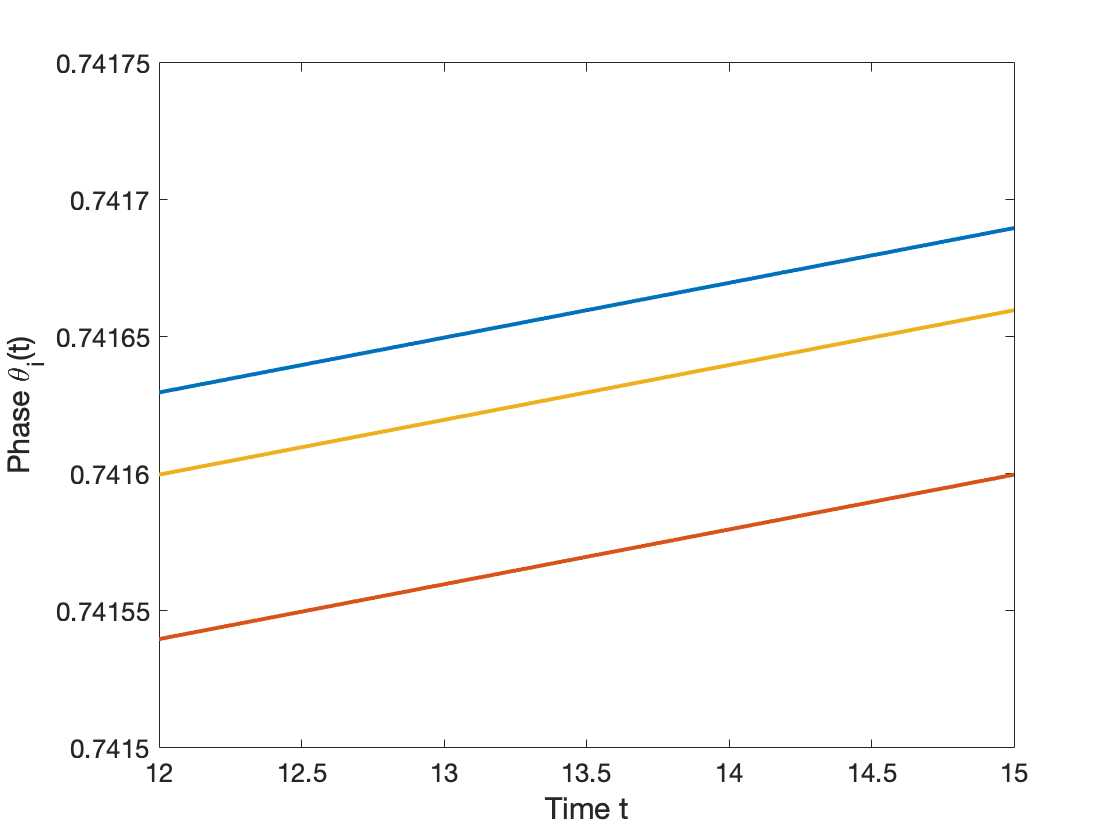}\label{Fig1:local_phs1}}
}\\

\mbox{
\subfigure[Frequency synchronization]{\includegraphics[width=0.48\textwidth]{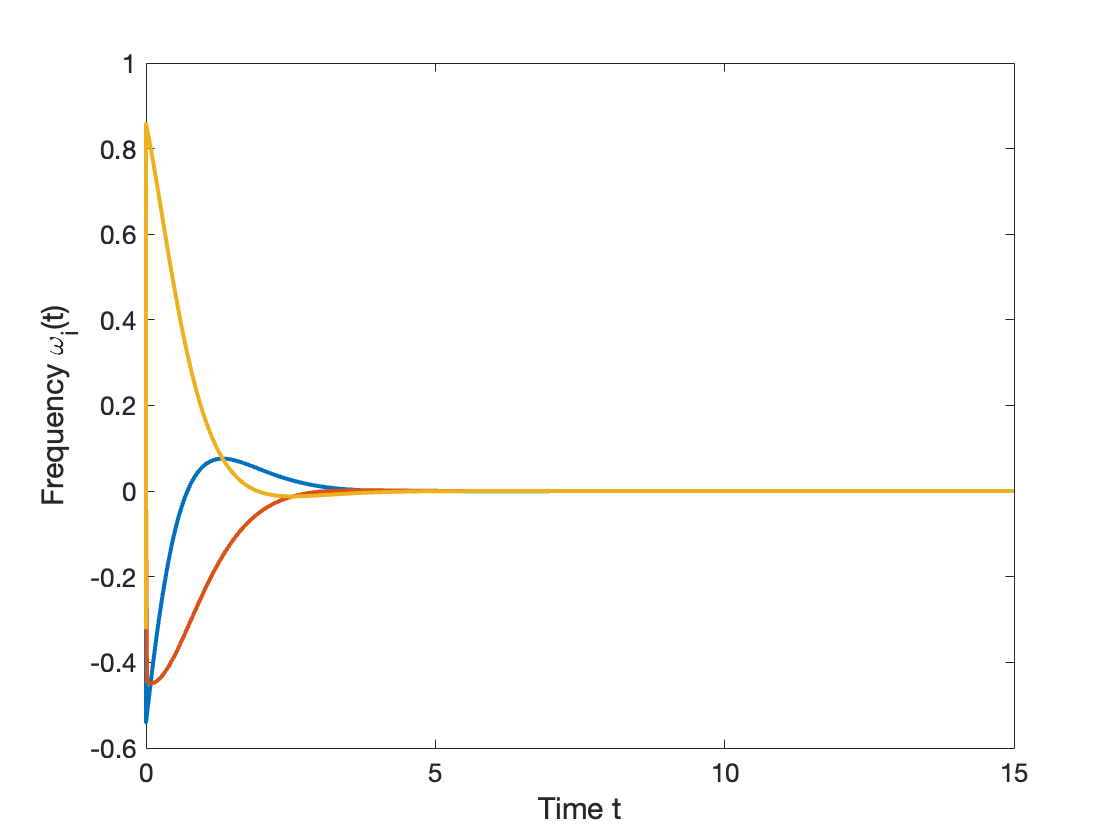}\label{Fig1:fre1}}

\subfigure[Frequency synchronization on local time interval]{\includegraphics[width=0.48\textwidth]{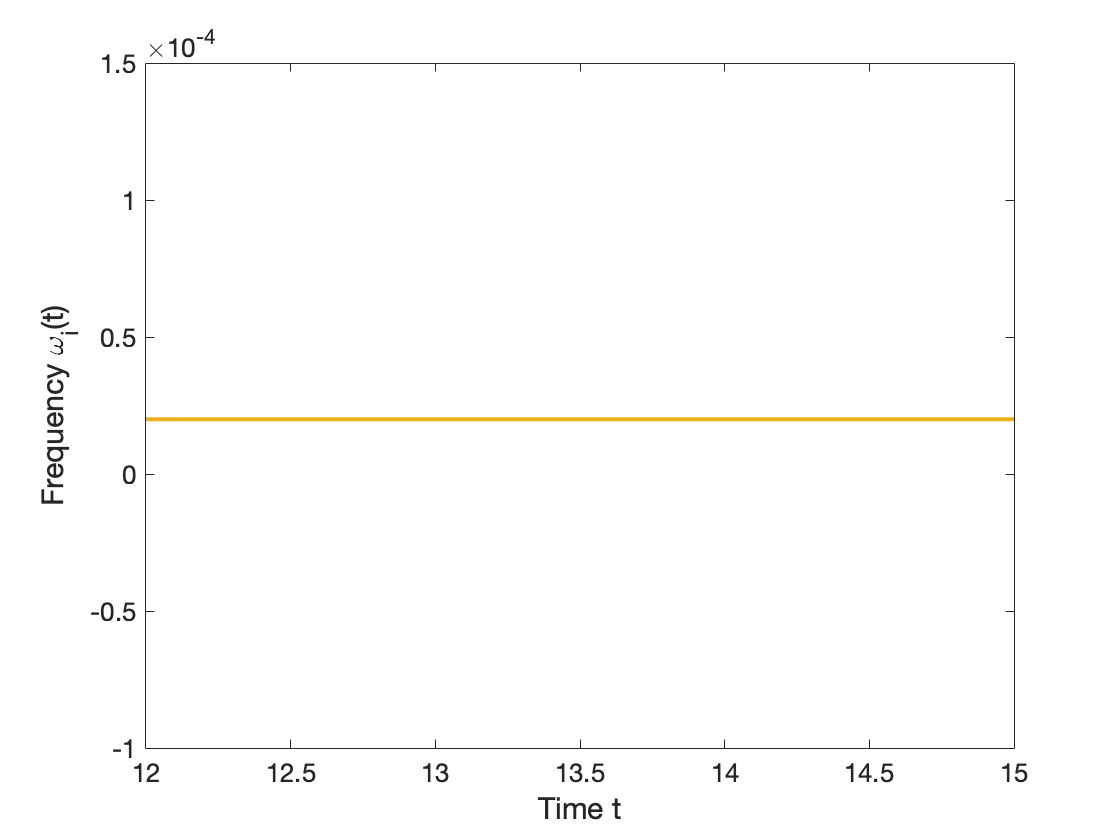}\label{Fig1:local_fre1}}
}

\caption{Complete synchronization for the second-order model \eqref{KMI} on the network \eqref{P-1}.}
\label{Fig1}
\end{figure}

\section{Summary}\label{sec:6}

In this paper, we investigated the synchronization problem of the Kuramoto model with inertia and frustration on a strongly connected digraph. The asymmetric network and the presence of frustration term imply that the standard methods for symmetric models cannot be directly applied. To address these difficulties, we developed crucial energy functions in the form of convex combinations, which are helpful to govern the phase and frequency diameters. Owing to the dissipative dynamics of these energy functions, we showed that the complete synchronization emerges exponentially fast. The parametric condition for the synchronization is explicitly presented, which corresponds to a regime of large coupling strength and small inertia and frustration. Improving the estimate of the synchronization condition remains an interesting problem, and a more general digraph will be considered in the future work.

\section*{Acknowledgments}

The work of T. Zhu is supported by the National Natural Science Foundation of China (Grant No. 12201172), the Scientific Research Foundation of Universities of Anhui Province of China (Grant No. 2022AH051790) and the Talent Research Fund of Hefei University of China (Grant No. 21-22RC23). The work of X. Zhang is supported by the National Natural Science Foundation of China (Grant No. 12471213) and the Talent Research Startup Fund of Wuhan University of China (Grant No. 2025-1301-015).

\end{document}